\documentclass[11pt]{article}

\usepackage{times}
\usepackage{amsmath,amsfonts,amstext,amssymb,amsbsy,amsopn,amsthm,eucal}
\usepackage{txfonts}
\usepackage{dsfont}
\usepackage{graphicx}   

\newcommand{\Ric}{\text{Ric}}
\newcommand{\Vol}{\text{Vol}}

\newcommand{\diam}{\text{diam}}
\newcommand{\CC}{\mathds{C}}

\newcommand{\RR}{\mathds{R}}
\newcommand{\Sn}{\mathds{S}}

\newcommand{\cC}{\mathcal{C}}
\newcommand{\cD}{\mathcal{D}}
\newcommand{\cF}{\mathcal{F}}

\newcommand{\cB}{\mathcal{B}}
\newcommand{\cM}{\mathcal{M}}

\begin{document}

\newtheorem{theorem}{Theorem}[section]

\newtheorem{proposition}{Proposition}[section]

\newtheorem{lemma}{Lemma}[section]

\newtheorem{corollary}{Corollary}[section]

\newtheorem{conjecture}{Conjecture}[section]

\theoremstyle{definition}
\newtheorem{definition}{Definition}[section]

\theoremstyle{remark}
\newtheorem{remark}{Remark}[section]

\theoremstyle{remark}
\newtheorem{example}{Example}[section]

\theoremstyle{remark}
\newtheorem{note}{Note}[section]

\theoremstyle{remark}
\newtheorem{question}{Question}[section]

\newenvironment{myindentpar}[1]%
{\begin{list}{}%
         {\setlength{\leftmargin}{#1}}%
         \item[]%
}
{\end{list}}

\title{Characterization of Tangent Cones of Noncollapsed Limits with Lower Ricci Bounds and Applications}

\author{Tobias Holck Colding and Aaron Naber\thanks{Department of Mathematics, Massachusetts Institute of Technology, Cambridge, MA 02139.  Emails: colding@math.mit.edu and anaber@math.mit.edu.
         The  first author
was partially supported by NSF Grants DMS  0606629, DMS 1104392,  and NSF FRG grant DMS
 0854774 and the second author by an NSF Postdoctoral Fellowship.}}


\date{\today}
\maketitle
\begin{abstract}
Consider a limit space  $(M_\alpha,g_\alpha,p_\alpha)\stackrel{GH}{\rightarrow} (Y,d_Y,p)$, where the $M_\alpha^n$ have a lower Ricci curvature bound and are volume noncollapsed.  The tangent cones of $Y$ at a point $p\in Y$ are known to be metric cones $C(X)$, however they need not be unique.  Let $\overline\Omega_{Y,p}\subseteq\cM_{GH}$ be the closed subset of compact metric spaces $X$ which arise as cross sections for the tangents cones of $Y$ at $p$.  In this paper we study the properties of $\overline\Omega_{Y,p}$.  In particular, we give necessary and sufficient conditions for an open smooth family $\Omega\equiv (X,g_s)$ of closed manifolds to satisfy $\overline\Omega =\overline\Omega_{Y,p}$ for {\it some} limit $Y$ and point $p\in Y$ as above, where $\overline\Omega$ is the closure of $\Omega$ in the set of metric spaces equipped with the Gromov-Hausdorff topology.  We use this characterization to construct examples which exhibit fundamentally new behaviors.  The first application is to construct limit spaces $(Y^n,d_Y,p)$ with $n\geq 3$ such that at $p$ there exists for every $0\leq k\leq n-2$ a tangent cone at $p$ of the form $\RR^{k}\times C(X^{n-k-1})$, where $X^{n-k-1}$ is a smooth manifold not isometric to the standard sphere.  In particular, this is the first example which shows that a stratification of a limit space $Y$ based on the Euclidean behavior of tangent cones is not possible or even well defined.  It is also the first example of a three dimensional limit space with nonunique tangent cones.  The second application is to construct a limit space $(Y^5,d_Y,p)$, such that at $p$ the tangent cones are not only not unique, but not homeomorphic.  Specifically, some tangent cones are homeomorphic to cones over $\CC P^2\sharp\overline{\CC P}^2$ while others are homeomorphic to cones over $\Sn^4$.
\end{abstract}

\section{Introduction}

In this paper we are interested in pointed Gromov-Hausdorff limits $(M_\alpha,g_\alpha,p_\alpha)\stackrel{GH}{\rightarrow} (Y,d_Y,p)$ such that the $M_\alpha$'s are $n$-dimensional and satisfy the lower Ricci bound
\begin{align}\label{con:Nonnegative_Ricci}
 \Ric(M_\alpha)\geq -(n-1)g\, ,
\end{align}
and the noncollapsing assumption
\begin{align}\label{con:noncollapsed}
 \Vol(B_1(p_\alpha))\geq v>0\, .
\end{align}

For any such limit $Y$, by Gromov's compactness theorem \cite{GLP,G}, any sequence $r_i\rightarrow 0$ contains a subsequence $r_j$ such that $(Y,r_j^{-1}d_Y,p)\stackrel{GH}{\rightarrow}(Y_p,d,p)$, where $Y_p$ is a length space.  Any such limit $Y_p$ is said to be a tangent cone of $Y$ at $p$.  By the noncollapsing assumption (\ref{con:noncollapsed}) it follows from \cite{ChC1}, \cite{ChC2} that any tangent cone must be a metric cone $Y_p\equiv C(X_p)$ over a compact metric space $X_p$ with $\diam\, X_p\leq \pi$ and Hausdorff dimension equal to $n-1$\footnote{Without the noncollapsing assumption tangent cones need not be metric cones by \cite{ChC2} and need not even be polar spaces by \cite{M4}.}.  However, by \cite{ChC2} tangent cones of $Y$ at $p$ need not be unique; cf. \cite{P2}.  More precisely, it may happen that there is a different sequence $\tilde r_j\rightarrow 0$ such that $(Y,\tilde r_j^{-1}d_Y,p)\stackrel{GH}{\rightarrow}(C(\tilde X_p),d,p)$ converges to a tangent cone $C(\tilde X_p)$ where $\tilde X_p$ and $X_p$ are not isometric.  We are therefore justified in defining for $p\in Y$ the family $\overline\Omega_{Y,p}\equiv \{X_s\}$ of metric spaces such that $C(X_s)$ arises as a tangent cone of $Y$ at $p$.

It is known that the family $\overline\Omega_{Y,p}\subseteq \cM_{GH}$, viewed as a subset of the space of all compact metric spaces endowed with the Gromov-Hausdorff topology, is compact and path connected.  It follows from \cite{ChC2} that the volume $\Vol (\cdot)$, or more precisely the $(n-1)$-dimensional Hausdorff measure, is independent of the cross section $X_s\in\overline\Omega_{Y,p}$ and is bounded from above by that of the round unit sphere of dimension $n-1$.  That is, 
\begin{align}\label{con:Volume_Constant}
 \Vol(X_s) = V\leq \Vol(S^{n-1}(1))\, .
\end{align}
Further, if $X_s\in\overline\Omega_{Y,p}$ is a smooth cross section, e.g. a smooth closed manifold, then because $\Ric(C(X_s))\geq 0$ we have that 
\begin{align}\label{con:Positive_Ricci}
 \Ric(X_s)\geq n-2\, .
\end{align}
In fact, it is fairly clear that (\ref{con:Positive_Ricci}) holds in the more general sense of \cite{LV}, \cite{S} even for singular $X_s$.  To fully understand the family $\overline\Omega_{Y,p}$ we introduce one more concept, that of Ricci closability.

\begin{definition}\label{d:closeable}
Let $(M^{n-1},g)$ be a smooth closed Riemannian manifold.  We say that $M$ is Ricci closable if for every $\epsilon>0$, there exists a smooth (open) pointed Riemannian manifold $(N^n_\epsilon,h_\epsilon,q_\epsilon)$ such that:
\begin{enumerate}
 \item $\Ric(N_\epsilon)\geq 0$.
 \item The annulus $A_{1,\infty}(q_\epsilon)\subseteq N_\epsilon$ is isometric to $A_{1,\infty}(C(M,(1-\epsilon)g))$.
\end{enumerate}
\end{definition}
\begin{remark}
Note that if the stronger condition that there exists $N$ with $\Ric(N)\geq 0$ and $A_{1,\infty}(q)\equiv A_{1,\infty}(C(M,g))$ holds, then $(M,g)$ is certainly Ricci closable.  Ricci closability acts as a form of geometric trivial cobordism condition.
\end{remark}

Now we ask the question:  
\begin{myindentpar}{1cm}
What subsets $\Omega\subseteq \cM_{GH}$ can arise as $\overline\Omega_{Y,p}$ for some limit space $Y$ coming from a sequence $M_\alpha\rightarrow Y$ which satisfies conditions (\ref{con:Nonnegative_Ricci}) and (\ref{con:noncollapsed})?
\end{myindentpar}
We have written down some basic necessary conditions on $\overline\Omega_{Y,p}$, and our main theorem is that these conditions are sufficent as well.

\begin{theorem}\label{t:Smooth_Section_Existence}
Let $\Omega$ be an open connected manifold, our parameter space.  Let $\{(X^{n-1},g_s)\}_{s\in\Omega}\subseteq \cM_{GH}$, with $n\geq 3$, be a smooth family of closed manifolds such that (\ref{con:Volume_Constant}) and (\ref{con:Positive_Ricci}) hold and such that for some $s_0$ we have that $X_{s_0}$ is Ricci closable.  Then there exists a sequence of complete manifolds $(M^n_\alpha,g_\alpha,p_\alpha)\stackrel{GH}{\rightarrow} (Y,d_Y,p)$ which satisfy (\ref{con:Nonnegative_Ricci}) and (\ref{con:noncollapsed}) for which $\overline{\{X_s\}} = \overline\Omega_{Y,p}$, where $\overline{\{X_s\}}$ is the closure of the set $\{X_s\}$ in the Gromov-Hausdorff topology.
\end{theorem}

\begin{remark}
 In fact, in the construction we will build the $M_\alpha$ to satisfy $\Ric(M_\alpha)\geq 0$.  Note here that $\Omega$, as a parameter space, is a smooth manifold which we are viewing as being embedded $\Omega\subseteq \cM_{GH}$ inside the space of metric spaces.
\end{remark}

In the applications we will be interested not so much in the smooth cones $C(X_s)$ which arise as tangent cones at $p\in Y$, but in the cones $C(X)$ where $X$ lies in the boundary of the closure $X\in \overline{\{X_s\}}\setminus \{X_s\}$.  There are two primary examples we will be interested in constructing through Theorem \ref{t:Smooth_Section_Existence}.  First, we will construct an example of a limit space $(Y,d_Y,p)$ such that at $p\in Y$ tangent cones are highly nonunique, and in fact, for every $0\leq k\leq n-2$ we can find a tangent cone that splits off precisely an $\RR^{k}$ factor.  Note this is in distinct contrast to the $\RR^n$ case, where if one tangent cone at a point is $\RR^n$, then so are all the other tangent cones at that point, see \cite{C}\footnote{For a limit of a sequence that collapses the situation is quite different, see \cite{M2}.}.  Note that if a tangent cone splits off an $\RR^{n-1}$ factor, then by \cite{ChC2} it is actually a $\RR^n$ factor, so that the nonunique splitting of $\RR^{k}$ factors for every $0\leq k\leq n-2$ is the most degenerate behavior one can get at a single point.  More precisely we have the following:

\begin{theorem}\label{t:Example_Rk_splitting}
For every $n\geq 3$, there exists a limit space $(M_\alpha^n,g_\alpha,p_\alpha)\stackrel{GH}{\rightarrow} (Y,d_Y,p)$ where each $M_\alpha$ satisfy (\ref{con:Nonnegative_Ricci}) and (\ref{con:noncollapsed}), and such that for each $0\leq k\leq n-2$, there exists a tangent cone at $p$ which is isometric to $\RR^k\times C(X)$, where $X$ is a smooth closed manifold not isometric to the standard sphere.
\end{theorem}

This example has the, potentially unfortunate, consequence that a topological stratification of a limit space $Y$ in the context of lower Ricci curvature can't be done based on tangent cone behavior alone.  This should be contrasted to the case of Alexandrov spaces, see \cite{P3}.  This also gives an example of a three dimensional limit space with nonunique tangent cones.

Our next example is of a limit space $(Y,d_Y,p)$, such that at $p\in Y$ there exist distinct tangent cones which are not only not isometric, but they are not even homeomorphic.  More precisely we have:

\begin{theorem}\label{t:Example_nonhomeomorphic}
There exists a limit space $(M_\alpha^5,g_\alpha,p_\alpha)\stackrel{GH}{\rightarrow} (Y^5,d_Y,p)$ of a sequence $M_\alpha$ satisfying (\ref{con:Nonnegative_Ricci}) and (\ref{con:noncollapsed}), and such that there exists distinct tangent cones $C(X_0)$, $C(X_1)$ at $p\in Y$ with $X_0$ homeomorphic to $\mathds{C}P^2\sharp \overline{\mathds{C}P}^2$ and $X_1$ homeomorphic to $\text{S}^4$.
\end{theorem}

Both of the last two theorems have analogues for tangent cones at infinity of open manifolds with nonnegative Ricci curvature and Euclidean volume growth.  We say that an open $n$-dimensional manifold with nonnegative Ricci curvature has Euclidean volume growth if for some $p\in M$ (hence all $p\in M$) there exists some $v>0$ such that for all $r>0$ we have that $\Vol (B_r(p))\geq v\,r^n$.  

\begin{theorem}
 We have the following:
\begin{enumerate}
 \item For $n\geq 3$, there exists a smooth open Riemannian manifold $(M^n,g)$ with $\Ric\geq 0$ and Euclidean volume growth such that for each $0\leq k\leq n-2$ one tangent cone at infinity of $M$ is isometric to $\RR^k\times C(X)$, where $X$ is a smooth closed manifold not isometric to the standard sphere.
\item There exists a smooth open Riemannian manifold $(M^5,g)$ with $\Ric\geq 0$ and Euclidean volume growth that has distinct tangent cones at infinity $C(X_0)$ and $C(X_1)$ with $X_0$ homeomorphic to $\mathds{C}P^2\sharp \overline{\mathds{C}P}^2$ and $X_1$ homeomorphic to $\text{S}^4$.
\end{enumerate}
\end{theorem}

Related to the above examples we conjecture the following:

\begin{conjecture}
Let $Y^n$ be a noncollapsed limit of Riemannian manifolds with lower Ricci bounds.  Let $\mathcal{NU}\subseteq Y$ be the set of points where the tangent cones at the given point are not unique, then $\dim_{Haus}(\mathcal{NU})\leq n-3$.
\end{conjecture}

\begin{conjecture}
Let $Y^n$ be a noncollapsed limit of Riemannian manifolds with lower Ricci bounds.  Let $\mathcal{NH}\subseteq Y$ be the set of points where the tangent cones at the given point are not of the same homeomorphism type, then $\dim_{Haus}(\mathcal{NH})\leq n-5$.
\end{conjecture}

In particular, we believe that for a four dimensional limit at each point tangent cones should be homeomorphic.

\vskip2mm
Finally, we mention that \cite{CN1} and \cite{CN2} contains some related results. In particular, in \cite{CN2} we will use some of the constructions of this paper. 

\section{Proof of Theorem \ref{t:Smooth_Section_Existence}}\label{s:ExamplesI}

The main technical lemma in the proof of Theorem \ref{t:Smooth_Section_Existence} is the following.

\begin{lemma}\label{l:ExamTech}
Let $X^{n-1}$ be a smooth compact manifold with $g(s)$, $s\in (-\infty,\infty)$, a family of metrics with $h_\infty<1$ such that:
\begin{enumerate}
\item $\Ric[g(s)]\geq (n-2)g(s)$.
\item $\frac{d}{ds}dv(g(s))=0$, where $dv$ is the associated volume form.
\item $|\partial_{s}g(s)|,|\partial_{s}\partial_{s}g(s)|\leq 1$ and $|\nabla\partial_{s}g(s)|\leq 1$, where the norms are taken with respect to $g(s)$.
\end{enumerate}

Then there exist functions $h:\mathds{R}^{+}\rightarrow(0,1)$ and $f:\RR^{+}\rightarrow(-\infty,\infty)$ with $lim_{r\rightarrow 0}h(r)=1$, $lim_{r\rightarrow \infty}h(r)=h_\infty$, $lim_{r\rightarrow 0}f(r)=-\infty$, $lim_{r\rightarrow \infty}f(r)=\infty$ and $lim_{r\rightarrow 0,\infty}rf'(r)=0$ such that the metric $\bar{g}=dr^{2}+r^{2}h^{2}(r)g(f(r))$ on $(0,\infty)\times X$ satisfies $\Ric[\bar{g}]\geq 0$.

Further if for some $T\in (-\infty,\infty)$ we have that $g(s)=g(T)$ for $s\leq T$ then we can pick $h$ such that for $r$ sufficiently small $h(r)\equiv 1$.
\end{lemma}

\begin{proof}
We only concern ourselves with the construction of $f$ and $h$ for $r\in (0,1)$.  Extending the construction for large $r$ is the same.

Now first we note that if $\bar{g}=dr^{2}+r^{2}h^{2}(r)g(f(r))$ as above then the following equations hold for the Ricci tensor, where the primes represent $r$ derivatives.
\begin{align}\label{e:R_rr}
 \overline{\Ric}_{rr} = -(n-1)\frac{(rh)''}{rh}+\frac{1}{4}g^{ab}g^{pq}g'_{ap}g'_{bq} -\frac{(rh)'}{rh}g^{ab}g'_{ab}-\frac{1}{2}g^{ab}g''_{ab}.
\end{align}
\begin{align}
 \overline{\Ric}_{ir} = \frac{1}{2}[\partial_{a}(g^{ab}g'_{bi}) - \partial_{i}(g^{ab}g'_{ab}) + \frac{1}{2}(g^{ab})'(\partial_{i}g_{ab} - g_{ib}g^{pq}\partial_{a}g_{pq})]
\end{align}
\begin{align}
\overline{\Ric}_{ij} = \Ric_{ij} + r^{2}h^{2}[(-(n-2)(\frac{(rh)'}{rh})^{2} -\frac{(rh)''}{rh}- \frac{1}{2}g^{ab}g'_{ab})g_{ij} \notag
\\
+(-\frac{n}{2}\frac{(rh)'}{rh} -\frac{1}{4}g^{ab}g'_{ab})g'_{ij} +\frac{1}{2}g^{ab}g'_{ai}g'_{bj}]
\end{align}

In the estimates it will turn out that terms involving either second derivatives of $g$ or products of first derivatives of $h$ and $g$ cannot be controlled in general.  Luckily the constant volume form tells us that
$$g^{ab}g'_{ab}=0\, ,$$
and by taking the $r$ derivative we get that
$$g^{ab}g''_{ab} = g^{ab}g^{pq}g'_{ap}g'_{bq}\, .$$

When we substitute these into (\ref{e:R_rr}) above we get
\begin{align}
\overline{\Ric}_{rr} = -(n-1)\frac{(rh)''}{rh} -\frac{1}{4} g^{ab}g^{pq}g'_{ap}g'_{bq}\, ,
\end{align}
similar substitutions may be made for the other equations.

Now for positive numbers $E,F\leq 1$ to be chosen define the functions
\begin{align}
 h(r)=1-\epsilon(r)=1-\frac{E}{\log(-\log(r_0r))}
\end{align}
and
\begin{align}
f(r)=-F\log(\log(-\log(r_0r)))\, ,
\end{align}
for $r\leq r_0$ to be chosen.  The following computations are straight forward:
\begin{align}
 \epsilon(r)=\frac{E}{\log(-\log (r_0r))}, \epsilon'(r)=\frac{E}{(\log(-\log (r_0r)))^{2}(-\log (r_0r))r},\notag
\\
\epsilon''(r)=\frac{E(-1+\frac{1}{(-\log (r_0r))}+\frac{2}{\log (-\log (r_0r))(-\log (r_0r))})} {(\log (-\log (r_0r)))^{2}(-\log (r_0r))r^{2}}
\end{align}
and so
\begin{align}
 \frac{(rh)'}{rh}=(\frac{1}{r}-\frac{\epsilon'}{1-\epsilon}) = \frac{1}{r}(1-\frac{E}{(1-\epsilon)(\log (-\log (r_0r)))^{2}(-\log (r_0r))})\leq \frac{1}{r}\, ,
\end{align}
\begin{align}
 \frac{(rh)''}{rh} = (-\frac{\epsilon''}{1-\epsilon} - \frac{2\epsilon'}{r(1-\epsilon)}) = \frac{-E(1+\frac{1}{(-\log (r_0r))}+\frac{2}{\log (-\log (r_0r))(-\log (r_0r))})} {(\log (-\log (r_0r)))^{2}(-\log (r_0r))r^{2}(1-\epsilon)} \notag\\
= -\frac{E}{2(\log (-\log (r_0r)))^{2}(-\log (r_0r))r^{2}}\, ,
\end{align}
where the last inequality holds for $r\leq 1$ and $r_0$ sufficiently small.  Also by our assumptions on $g(s)$ we have that $|g'|\leq |f'|\leq \frac{F}{\log (-\log (r_0r))(-\log (r_0r))r}$.  Finally, if we plug all of this into our equations for the Ricci tensor we get, where $D=D(n)$ is a dimensional constant:
\begin{align}\label{e:Rc1}
\overline{\Ric}_{rr}\geq \frac{E}{(\log (-\log (r_0r)))^{2}(-\log (r_0r))r^{2}}- \frac{DF^{2}}{(\log (-\log (r_0r)))^{2}(-\log (r_0r))^{2}r^{2}} \notag
\\
\geq \frac{E}{2(\log (-\log(r_0r)))^{2}(-\log (r_0r))r^{2}}\, ,
\end{align}
\begin{align}\label{e:Rc2}
\overline{\Ric}_{ir}\geq \frac{-DF}{\log (-\log(r_0r))(-\log (r_0r))r}\, ,
\end{align}
\begin{align}\label{e:Rc3}
\overline{\Ric}_{ii}\geq r^{2}h^{2}[\frac{(n-2)\epsilon}{r^{2}h^{2}} + \frac{E}{2(\log (-\log (r_0r)))^{2}(-\log (r_0r))r^{2}} - \frac{DF}{\log (-\log (r_0r))(-\log (r_0r))r^{2}}\notag\\
-  \frac{DF^{2}}{(\log (-\log (r_0r)))^{2}(-\log (r_0r))^{2}r^{2}}]\geq r^{2}h^{2}\frac{E}{\log (-\log (r_0r))r^{2}}\, ,
\end{align}
where the last inequalities on (\ref{e:Rc1}) and (\ref{e:Rc3}) require $E\geq E(n,F)$ and $r_0$ sufficiently small.  Now it is clear from the above that we get positive Ricci in the $r$ and $M$ directions.  The difficulty is that we have a mixed term (\ref{e:Rc2}) which can certainly be negative and in fact dominates the positivity of (\ref{e:Rc1}). To see positivity fix a point $(r,x)\in (0,1)\times M$ and assume at this point $g_{ij}(f(r))=\delta_{ij}$.  Then every unit direction at this point is of the form $\delta \hat{r}+\frac{\sqrt{1-\delta^{2}}}{rh}\hat{i}$ for $\delta\in [0,1]$ and we can compute:
\begin{align}
 \overline{\Ric}_{(\delta r+\frac{\sqrt{1-\delta^{2}}}{rh}i)(\delta r+\frac{\sqrt{1-\delta^{2}}}{rh}i)}
 \geq \frac{1}{\log (-\log (r_0r))r^{2}}[\frac{E\delta^{2}}{2\log (-\log (r_0r))(-\log (r_0r))}\notag\\
- \frac{2DF\delta\sqrt{1-\delta^{2}}}{(-\log (r_0r))h} + E(1-\delta^{2})]
\end{align}
\begin{align}
\geq \frac{1}{2\log (-\log (r_0r))r^{2}}[\frac{E\delta^{2}}{\log (-\log (r_0r))(-\log(r_0r))} - \frac{DF\delta\sqrt{1-\delta^{2}}}{(-\log(r_0r))} + E(1-\delta^{2})\, ,
\end{align}
where the last inequality is for $r\leq 1$ and after possibly changing $D$.  To see this is positive for any $\delta\in [0,1]$ we break it into two cases, when $\sqrt{1-\delta^{2}}\geq \frac{1}{(-\log (r_0r))}$ and $\sqrt{1-\delta^{2}}\leq \frac{1}{(-\log (r_0r))}$.  For the first case we see that
\begin{align}
 \overline{\Ric}_{(\delta r+\frac{\sqrt{1-\delta^{2}}}{rh}i)(\delta r+\frac{\sqrt{1-\delta^{2}}}{rh}i)} \geq \frac{\sqrt{1-\delta^{2}}}{\log (-\log (r_0r))r^{2}}[\frac{-DF}{(-\log(r_0r))}+\frac{E}{(-\log(r_0r))}] \geq 0\, ,
\end{align}
for $E\geq DF$.  For the case $\sqrt{1-\delta^{2}}\leq \frac{1}{(-\log(r_0r))}$ we first note that $\delta\geq \frac{1}{2}$ for $r\leq 1$ and then group the first two terms to get:
\begin{align}
 \overline{\Ric}_{(\delta r+\frac{\sqrt{1-\delta^{2}}}{rh}i)(\delta r+\frac{\sqrt{1-\delta^{2}}}{rh}i)} \geq \frac{\delta}{\log (-\log (r_0r))(-\log (r_0r))r^{2}}[\frac{E}{2\log (-\log (r_0r))}-\frac{DF}{(-\log (r_0r))}] \geq 0
\end{align}
for $E\geq DF$ and $r\leq 1$, and $r_0$ sufficiently small as claimed.

Now extending $f$ and $h$ to the rest of $r$ can be done in the same manner, and handling the case when $g(s)=g(T)$ stabilizes is comparatively simple and can be done with a cutoff function so that $h(r)$ is concave in this region.  Note for any $h_\infty$ we can pick $F$, and hence $E$, sufficiently small as to make the volume loss as small as we wish.
\end{proof}

With the above in hand it is easy to finish Theorem \ref{t:Smooth_Section_Existence}.

\begin{proof}[Proof of Theorem \ref{t:Smooth_Section_Existence}]
We begin by constructing what will be the limit space $Y=C(X)$ of the theorem.  Let $c:(-\infty,\infty)\rightarrow\Omega$ be a smooth map such that for every open neighborhood $U\subseteq \Omega$ there are $t_a\rightarrow\infty$ such that $c(-t_a)=c(t_a)\in U$.  

In the case when condition (\ref{con:Volume_Constant}) is assumed we can apply a theorem of Moser \cite{Mo}, which tells us that for a compact manifold $X$ if $w_0,w_1$ are volume forms with the same volume then there exists a diffeomorphism $\phi:X\rightarrow X$ such that $w_1=\phi^*w_0$.  With this in mind there is no loss in assuming that for each $s,t\in (-\infty,\infty)$ we have $dv_{g(c(s))}=dv_{g(c(t))}$, since the other conditions of the theorem are diffeomorphism invariant.

Because $g(x)$ is smooth for $x\in \Omega$ we can be sure, after possibly reparametrizing $c$, that $g(t)\equiv g(c(t))$ satisfies Lemma \ref{l:ExamTech}.  We take
$$\bar{g}=dr^{2}+r^{2}h^{2}(r)g(f(r))$$
from this lemma.  The conditions on $h$ guarantee that the metric extends to a complete metric on the cone $Y$.

Now we argue that $Y$ satisfies the conditions of the theorem, hence for each $s\in \bar\Omega$ that the metric cone $C(X_s)$ is realized as a tangent cone of $Y$.  So let $r_a\rightarrow 0$ such that $c(f(r_a))\rightarrow s$, which we can do by the conditions on $f$ and the construction of $c$.  If we consider the rescaled metric
$$r^{-2}_a\bar{g}\approx dr^{2}+r^{2}h^{2}(r_ar)g(f(r_a r))\, ,$$
then by the condition $\lim_{r\rightarrow 0} rf'(r)=0$ we see that this converges to the desired tangent cone as claimed.

Finally, we wish to show that if for some $s_0\in \Omega$ that if $X_{s_0}$ is Ricci closable, then $(Y,d)$ can be realized as a limit $(M_\alpha,g_\alpha,p_\alpha)$ of Riemannian manifolds with nonnegative Ricci curvature.  For each $\alpha$ let $c_\alpha(t)$ be a smooth curve such that
\begin{align}
c_\alpha(t) = \left\{ \begin{array}{rl}
 c(t)&\mbox{ if } t\geq -\alpha\notag\\
  s_0 &\mbox{ if } t\leq -2\alpha
       \end{array} \right. \, .
\end{align}

For each $\alpha$ let $(C(X),d_\alpha)$ be the metric space associated with the curve
$$
g_\alpha(t)\equiv (1-\alpha^{-1})g(c_\alpha(t))\, ,
$$
as by Lemma \ref{l:ExamTech} (again, if need be we can reparametrize $c_\alpha(t)$ for $t<-\alpha$ to force $g_\alpha(t)$ to satisfy the requirements of the Lemma).  Near the cone point we have that $(C(X),d_\alpha)$ is isometric to $C(X,(1-\frac{1}{\alpha})g(s_0))$.  By the assumption of Ricci closability there exists a complete Riemannian manifold $(N_\alpha,h_\alpha,p_\alpha)$ such that
$$
\Ric(N_\alpha)\geq 0\, ,
$$
and
$$
A_{1,\infty}(p_\alpha)\equiv A_{1,\infty}(C(M,(1-\alpha^{-i})g(s_0)))\, .
$$
Thus we can glue these together to construct smooth Riemannian manifolds $(M_\alpha,g_\alpha,p_\alpha)$.  This is our desired sequence.
\end{proof}

\section{Example I}\label{ss:Example1}
Our first application of Theorem \ref{t:Smooth_Section_Existence} is to provide, for $n\geq 3$, examples of limit spaces
\begin{align}
 (M^n_\alpha,g_\alpha,p_\alpha)\stackrel{GH}{\rightarrow} (Y^n,d_Y,p)\, ,
\end{align}
where each $M_\alpha$ has nonnegative Ricci curvature with $\Vol(B_1(p_\alpha))>v>0$, and such that at $p\in Y$ the tangent cones are not only nonunique, but for each $0\leq k\leq n-2$ we can find a sequence $r^k_a\rightarrow 0$ such that
\begin{align}
(Y,(r^k_a)^{-1}d_Y,p)\stackrel{GH}{\rightarrow} \RR^k\times C(X^{n-k-1})\, ,
\end{align}
where the $X^{n-k-1}$ are smooth manifolds with $\Vol(X^{n-k-1})<\Vol(\Sn^{n-k-1})$.  That is, for each $0\leq k\leq n-2$ we can find a tangent cone which splits off precisely an $\RR^k$ factor.  As was remarked earlier this is optimal, in that if any tangent cone were to split a $\RR^{n-1}$-factor, then by \cite{ChC2} we would have that $p$ is actually a regular point of $Y$, and in particular by \cite{C} every tangent cone would be $\RR^n$.

To construct our example we will build a family of smooth manifolds $(S^{n-1},\bar g_s)$ , and apply Theorem \ref{t:Smooth_Section_Existence}.  To describe this family let us first define for $0<t\leq 1$ the $t$-suspension, $S_t(X)$, over a smooth manifold $X$.  That is, for $0<t\leq 1$ and a smooth manifold $X$, the metric space $S_t(X)$ is homemorphic to the suspension over $X$ and its geometry is defined by the metric
$$
dr^2+\sin^2(\frac{1}{t} r)\,d_X^2\, ,
$$
for $r\in (0,t\pi)$.  Notice then that $S_1(X)$ is the standard metric suspension of $X$.  Now for any $\vec t \in \cD\equiv \{\vec t\in \RR^{n-1}: 0<t_{n-1}\leq t_{n-2}\leq\ldots\leq t_{1}\leq 1\}$ we can define the metric
$$
g_{\vec t}\equiv S_{t_{1}}(\ldots S_{t_{n-2}}(\Sn^1(t_{n-1})))\, ,
$$
where $\Sn^1(t_{n-1})$ is the circle of radius $t_{n-1}$.  Note in particular that $g_{(t,\ldots,t)}$ is the $n-1$ sphere of radius $t$.  More generally, we have that $g_{(1,\ldots,1,t,\ldots,t)}$, where the first $k$ entries are $1$, is isometric to the $k$-fold suspension of the $n-k-1$ sphere of radius $t$.  This tells us in particular that
$$
C((\Sn^{n-1},g_{(1,\ldots,1,t,\ldots,t)}))\equiv \RR^k\times C(\Sn^{n-k-1}(t))\, .
$$

 Let us define the subset $\Omega\subseteq \RR^{n-1}$ by the condition 
$$
\Omega\equiv \{\vec t\in\RR^{n-1}:0<t_{n-1}\leq t_{n-2}\leq\ldots\leq t_{1}< 1\text{ and } \Vol(g_{\vec t})=\Vol(g_{\frac{1}{2},\ldots,\frac{1}{2}})\}\, .
$$
We have that $\Omega$ satisfies the following basic properties:
\begin{enumerate}
 \item $\Omega$ is a smooth, connected, open submanifold of dimension $n-2$.
 \item $(\frac{1}{2},\ldots,\frac{1}{2})\in\Omega$.
 \item For each $0\leq k\leq n-2$ $\exists$ $0<t_k<1$ and $\vec t_i\in\Omega\to (1,\ldots,1,t_k,\ldots,t_k)$ such that $(\Sn^{n-1},g_{\vec t_i})\stackrel{GH}{\rightarrow} (\Sn^{n-1},g_{(1,\ldots,1,t_k,\ldots,t_k)})$, where the first $k$ entries are $1$.
\end{enumerate}

Now the collection $g_s$ with $s\in \Omega$ almost defines our family.  Notice in particular that since $g_{(\frac{1}{2},\ldots,\frac{1}{2})}$ is the $n-1$ sphere of radius $\frac{1}{2}$ it is certainly Ricci closable, and that for every $0\leq k\leq n-2$ we have by the third condition above that $\RR^k\times C(\Sn^{n-k-1}(t_k))\in \overline{g(\Omega)}$, where the closure is in the Gromov-Hausdorff sense.  The remaining issue is simply that our metrics $g_s$ on $\Sn^{n-1}$ are not smooth.  However, for $\vec t\in\Omega$ they do satisfy 
$$\sec[g_{\vec t}]> 1+\epsilon(\vec t)\, ,$$
both on the smooth part and in the Alexandrov sense on the whole, where $\epsilon(\vec t)\to 0$ as $\vec t\to\partial\Omega$.  Although not smooth, the singularities are isometric spheres and may be easily smoothed in a canonical fashion by writing in normal coordinates with respect to the singular spheres, see \cite{P1}, \cite{M1}, \cite{M3} for instance.  We let $\bar g_{\vec t}$ be such a smoothing, where for each $\vec t$ we can then easily arrange, by smoothing a sufficiently small amount, that 
\begin{align}
\sec[\bar g_{\vec t}]>1+\frac{1}{2}\epsilon(\vec t)\,
\end{align}
while 
\begin{align}
 |\Vol(\bar g_{\vec t})-\Vol(g_{\vec t})|<\delta(\vec t)\, ,
\end{align}
where $\delta(\vec t)<<\epsilon(\vec t)$.  Thus, after a slight rescaling of each $\bar g_{\vec t}$, we can guarantee that the volumes continue to coincide and that $\sec_{\vec t}\geq 1$ for $s\in\Omega$.  This family thus satisfies Theorem \ref{t:Smooth_Section_Existence}, and we can construct the desired limit space $(M^n_\alpha,g_\alpha,p_\alpha)\rightarrow (Y^n,d_Y,p)$ as in the Theorem.

\section{Example II}\label{ss:Example3}

In this section we present one further example of interest.  We wish to construct a complete limit space
\begin{align}
(M^{5}_{\alpha},g_\alpha,p_\alpha)\rightarrow (Y^5,d_Y,p)\, ,
\end{align}
where each $M_\alpha$ satisfy $\Ric_\alpha\geq 0$, $\Vol(B_1(p_\alpha))\geq v>0$, and such that at $p$ the tangent cones of $Y$ are not only not unique, but there exist distinct tangent cones which are not even homeomorphic.  Specifically there are sequences $r_a\rightarrow 0$ and $r'_a\rightarrow 0$ with
\begin{align}
(Y,r_a^{-1}d_Y,p)\rightarrow (C(X_p),d_{Y_p},p)\, ,\notag\\
(Y,r_a'^{-1}d_Y,p)\rightarrow (C(X_p'),d_{Y_p'},p)\, ,
\end{align}
and such that homeomorphically we have
\begin{align}
X_p\approx \mathds{C}P^2\sharp \overline{\mathds{C}P}^2\, ,\notag\\
X_p'\approx \Sn^4\, .
\end{align}

To construct our example we wish to again use Theorem \ref{t:Smooth_Section_Existence}.  We will construct a family of metrics $(\mathds{C}P^2\sharp \mathds{C}P^2,g_t)$ with $t\in (0,2]$ which satisfy the hypothesis of the theorem and such that
$$\lim_{t\rightarrow 0}\,(\mathds{C}P^2\sharp \overline{\CC P}^2,g_t) = (\Sn^4,g_0)\, .$$
Geometrically, $(\Sn^4,g_0)$ will contain two singular points and will look roughly like a football.  On the other hand, $(\mathds{C}P^2\sharp \mathds{C}P^2,g_2)$ will have a sufficiently nice form that we will be able to show that it is Ricci closable.  Once this family is constructed we can immediately apply Theorem \ref{t:Smooth_Section_Existence} to produce our example.

The construction of the family will be done in several steps.  We begin by introducing our basic ansatz.  Let $\Sn^3$ be the three sphere, viewed as the Lie Group $SU(2)$, with the standard frame $X$, $Y$, $Z$ such that
$$
[X,Y]=2Z\,, \,[Y,Z]=2X,\, [Z,X]=2Y\, .
$$

Each piece of the various constructions will be a metric on $(r_0,r_1)\times \Sn^3$ which takes the form
\begin{align}
 dr^2 + A(r)^2dX^2 + B^2(r)\left(dY^2+dZ^2\right)\, ,
\end{align}
where $0\leq r_0<r_1\leq \frac{\pi}{2}$.  Notice that by employing various boundary data on $A$ and $B$ we can get these metrics to close up to smooth metrics on $\mathds{C}P^2$, $\mathds{C}P^2\sharp \overline{\mathds{C}P}^2$ or $\mathds{C}P^2\setminus \overline D^4$, where $\overline D^4$ is the closed $4$-ball.  The Ricci curvature of these metrics satisfy the equations
\begin{align}\label{e:ricci}
&\Ric(r,r) = -\frac{A''}{A}-2\frac{B''}{B}\, ,\\
&\frac{1}{|X|^2}\Ric(X,X) = -\frac{A''}{A}-2\frac{A'B'}{AB}+2\frac{A^2}{B^4}\, , \\
&\frac{1}{|Y|^2}\Ric(Y,Y) = -\frac{B''}{B}-\frac{A'B'}{AB}-\left(\frac{B'}{B}\right)^2+2\frac{2B^2-A^2}{B^4}\, , \\
&\frac{1}{|Z|^2}\Ric(Z,Z) = -\frac{B''}{B}-\frac{A'B'}{AB}-\left(\frac{B'}{B}\right)^2+2\frac{2B^2-A^2}{B^4}\, , 
\end{align}
with all other Ricci terms vanishing.

\subsection{Bubble Construction}

Our bubbles mimic those of \cite{P1}, see also \cite{M1}, \cite{M3}.  Let $0<b_0\leq 1$ be a constant which will be fixed at the end of the construction.  For each $0\leq \epsilon\leq 1$ let us consider the metric spaces $\cB^\epsilon$ defined by
\begin{align}
&A^\epsilon_\cB(r) \equiv b_0\frac{1}{2}\sin(2r)\, , \\
&B^\epsilon_\cB(r) \equiv b_0\left(\frac{1}{100}-(\frac{1}{2}-\frac{1}{100})\epsilon\right)\cosh(\frac{\epsilon}{100} r)\, ,
\end{align}
for $r\in (0,r_\epsilon]$, where $r_\epsilon$ is such that $A^\epsilon_\cB(r_\epsilon)=B^\epsilon_\cB(r_\epsilon)$.  Our bubbles $\cB^\epsilon$ are smooth manifolds with boundary which are homeomorphic to $\mathds{C}P^2\setminus\overline D^4$.  Notice that $0<r_1\leq r_\epsilon\leq r_0\equiv \frac{\pi}{4}$, and that for each such $\epsilon$ the boundary $\partial\cB_\epsilon$ is an isometric sphere of radius between $\frac{b_0}{100}$ and $\frac{b_0}{2}$.  The second fundamental forms of each boundary, $T(\partial\cB^\epsilon)$, are uniformly positive and satisfy the estimate
\begin{align}\label{e:sff}
 T(\partial\cB_\epsilon)>\lambda_\epsilon b_0\, ,
\end{align}
where $\lambda_\epsilon\to 0$ as $\epsilon\to 0$.  Further, the boundary $\partial\cB^0$ has zero second fundamental form, and two copies of $\cB^0$ may be glued to contruct a smooth metric on $\CC P^2\sharp\overline{\CC P}^2$.  Note for all $b_0$ sufficiently small, that by (\ref{e:ricci}) the Ricci curvatures of each of these spaces are uniformly positive independent of $\epsilon\in [0,1]$.

\vskip3mm
{\bf Step 1:}
\vskip3mm
Here we construct the metrics $(\CC P^2\sharp\overline{\CC P}^2,g_t)$ for $t\in (0,1]$.  The metrics will have the claimed property that as $t\to 0$, $(\CC P^2\sharp\overline{\CC P}^2,g_t)\to (\Sn^4,g_0)$.  We will show simply that the metrics satisfy 
$$\Vol_t>\eta>0\, $$
$$\Ric_t>\eta>0\, ,$$ 
independent of $t$.  It then holds that conditions (\ref{con:Volume_Constant}) and (\ref{con:Positive_Ricci}) can be forced after appropriate rescalings. 

For each $\ell>0$ we first consider the football metrics $\cF_\ell$ defined by
\begin{align}
&A^\ell_\cF(r) \equiv \frac{1}{2}\ell\sin(2r)\, , \\
&B^\ell_\cF(r) \equiv \frac{1}{2}\ell\sin(2r)\, ,
\end{align}
for $r\in (0,\frac{\pi}{2})$.  By definition we let $\cF_\ell(s)$ be the smooth manifold with boundary, homeomorphic to $[0,1]\times \Sn^3$, gotten by the restriction $r\in [s,\frac{\pi}{2}-s]$.  For all $\delta>0$ we can pick $\ell\leq \bar\ell(\delta)$ such that for all $0<s< \frac{\pi}{4}$ the boundary of $\cF_\ell(s)$ is a sphere of radius $\rho(s)$ and has a second fundamental form which satisfies $T(\partial\cF_\ell(s))>-\delta\rho(s)$.

Let us fix $\delta<<\lambda_{1}$, where $\lambda_1$ is as in (\ref{e:sff}), and correspondingly let $\ell\leq \bar\ell(\delta)$.  For all $0< t\leq 1$ let $g_t$ be the smooth metric on $\mathds{C}P^2\sharp \overline{\mathds{C}P}^2$ gotten by gluing $\cF^{\bar\ell}(t \frac{\pi}{4})$ with $\cB^1$ and then smoothing.  As in \cite{P1}, the constraints on the second fundamental forms guarantee that this smoothing can be done so that it preserve the positive Ricci curvature.  Because the smoothing is done with respect to normal coordinates on the boundary, see \cite{P1}, it is clear that this can be done smoothly in $t$, and that the Ricci curvature is uniformly positive independent of $0<t\leq 1$.  This follows because it holds for $\cF^{\ell}$, and near the bubble $\cB^1$ we have that $\Ric\to \infty$ as $t\to 0$.  Notice that the metric $(\mathds{C}P^2\sharp \overline{\mathds{C}P}^2,g_1)$ is now just a smoothing of two copies of $\cB^1$ glued along their boundaries.

\vskip3mm
{\bf Step 2:}
\vskip3mm

Here we construct the metrics $(\CC P^2\sharp\overline{\CC P}^2,g_t)$ for $t\in [1,2]$.  We will see later that the metric $(\CC P^2\sharp\overline{\CC P}^2,g_2)$ is Ricci closable.  Again, we will only worry about seeing that there exists uniform positive lower bounds on the volume and Ricci curvature.

Let us now consider the family of metrics $(\CC P^2\sharp\overline{\CC P}^2,g_t)$, $t\in [1,2]$, defined by gluing two copies of $\cB^{2-t}$ along the boundaries and smoothing.  Again, it follows from the conditions on the second fundamental forms and \cite{P1} that these metrics themselves have uniformly positive Ricci curvature.  Further, as we previously observed the metric space $(\CC P^2\sharp\overline{\CC P}^2,g_2)$ requires no smoothing, and with only a little care we see that the smoothing process can be done smoothly in $t$.

\subsection{Closability} 

Now that we have constructed the $1$-parameter family of metrics $(\CC P^2\sharp \overline{\CC P}^2,g_t)$ with $t\in (0,2]$, we need to show that at least one of these metrics is Ricci closable, see Definition \ref{d:closeable}.  A clear necessary condition for this is that the manifold in question be trivially cobordent, hence our choice of $\CC P^2\sharp \overline{\CC P}^2$.  We will focus on the space $(\CC P^2\sharp\overline{\CC P}^2,g_2)$, whose geometry is explicitly described by the conditions
\begin{align}\label{e:B0}
&A^0_\cB(r) \equiv \frac{b_0}{2}\sin(2r)\, , \\
&B^0_\cB(r) \equiv \frac{1}{2}b_0\, ,
\end{align}
with $r\in (0,\frac{\pi}{2})$.  We have viewed $\CC P^2\sharp \overline{\CC P}^2$ as the warped product $(0,\frac{\pi}{2})\times S^3$, where at the boundary ends the Hopf fiber collapses to glue in two $S^2$'s.  It will now be more convenient to visualize $\CC P^2\sharp \overline{\CC P}^2$ as the nontrivial $S^2$ bundle over $S^2$.  Topologically, the $5$-manifold which then realizes the trivial cobordism of $\CC P^2\sharp \overline{\CC P}^2$ can be viewed as a nontrivial $\bar D^3$ bundle over $S^2$, where $\bar D^3$ is the closed $3$-ball.  The geometric cobordism we will build on this space, which will satisfy Definition \ref{d:closeable}, will be built in two pieces.  These pieces will themselves then be glued together.  Our ansatz for the metric construction on each piece will look similar to before, though a little more complicated.  We consider metrics of the following form:
\begin{align}
 ds^2+C^2(s)dr^2+D^2(s)A^2(r)dX^2+E^2(s)B^2(r)\left(dY^2+dZ^2\right)\, ,
\end{align}
where $s\in (s_0,s_1)$, $r\in(0,\frac{\pi}{2})$, and $X,Y,Z$ are the standard left invariant vector fields on $S^3$ as before.  The Ricci curvature on such spaces takes the form
\begin{align}\label{e:Riccicomputations3}
&\Ric(s,s) = -\frac{\ddot C}{C}-\frac{\ddot D}{D}-2\frac{\ddot E}{E}\, ,\\
&\Ric(s,r) = \frac{\dot C}{C}\left(\frac{A'}{A}+2\frac{B'}{B}\right)-\frac{\dot D}{D}\frac{A'}{A}-2\frac{\dot E}{E}\frac{B'}{B}\, ,\\
&\frac{1}{|r|^2}\Ric(r,r) = -\frac{\ddot C}{C}-\frac{\dot C}{C}\left(\frac{\dot D}{D}+2\frac{\dot E}{E}\right)-C^{-2}\left(\frac{A''}{A}+2\frac{B''}{B}\right)\, ,\\
&\frac{1}{|X|^2}\Ric(X,X) = -\frac{\ddot D}{D}-C^{-2}\frac{A''}{A}-\frac{\dot D}{D}\left(\frac{\dot C}{C}+2\frac{\dot E}{E}\right)-2C^{-2}\frac{A'B'}{AB}+2\frac{D^2A^2}{B^4E^4}\, , \\
&\frac{1}{|Y|^2}\Ric(Y,Y) = -\frac{\ddot E}{E}-C^{-2}\frac{B''}{B}-\frac{\dot E}{E}\left(\frac{\dot C}{C}+\frac{\dot D}{D}+\frac{\dot E}{E}\right)-C^{-2}\frac{B'}{B}\left(\frac{A'}{A}+\frac{B'}{B}\right)+2\frac{2B^2E^2-A^2D^2}{B^4E^4}\, ,\\
&\frac{1}{|Z|^2}\Ric(Z,Z) = -\frac{\ddot E}{E}-C^{-2}\frac{B''}{B}-\frac{\dot E}{E}\left(\frac{\dot C}{C}+\frac{\dot D}{D}+\frac{\dot E}{E}\right)-C^{-2}\frac{B'}{B}\left(\frac{A'}{A}+\frac{B'}{B}\right)+2\frac{2B^2E^2-A^2D^2}{B^4E^4}\, ,
\end{align}
where all other Ricci terms vanish.

Our first piece of the geometric cobordism, which is a metric space we will denote by $\cC^1$, will be defined by the functions
\begin{align}
&A_2(r) \equiv \frac{b_0}{2}\sin(2r)\, , \\
&B_2(r) \equiv \frac{b_0}{2}\, ,\\
&C_2(s)=D_2(s)=E_2(s) \equiv s\, ,
\end{align}
with $s\in [1,\infty)$.  That is, $\cC^1$ is simply the top half of the cone over $(\CC P^2\sharp \overline{\CC P}^2,g_2)$.  To smooth this out near the cone point we consider the metric space $\cC^2$ defined by
\begin{align}
&A_3(r) \equiv \frac{b_1}{2}\sin(2r)\, , \\
&B_3(r) \equiv \frac{b_1}{2}\, , \\
&C_3(s)=D_3(s)\equiv \sin(2s)\, , \\
&E_3(s)\equiv e_0\cosh(e_0 s)\, ,
\end{align}
with $s\in (0,s_0)$, where $s_0$ defined by the condition $C_3(s_0)=E_3(s_0)$.  A computation using (\ref{e:Riccicomputations3}) tells us that for each $e_0$ sufficiently small that for $b_1$ sufficiently small we have $s_0>0$, and that the underlying space having strictly positive Ricci curvature.  Further, in analogy with the construction of $\cB^\epsilon$, we have that the boundary $\partial\cC^2$ has strictly positive second fundamental form,
$$
T(\partial\cC)>\lambda>0\, .
$$

The argument now mimicks that of {\bf Step 1}.  If we fix $b_0$ sufficiently small in comparison to $\lambda$, then the second fundamental form of the boundary of $\partial \cC^2$ is more positive than the second fundamental form of $\partial\cC^1$ is negative.  Thus, by using \cite{P1} once again and rescaling $\cC^1$ appropriately, we may glue $\cC^1$ with $\cC^2$ so that after smoothing we have a manifold with nonnegative Ricci curvature.  With $b_0$ chosen appropriate this then shows that $(\CC P^2\sharp \overline{\CC P}^2,g_2)$ is Ricci closable as claimed, and thus finishes the construction.


\begin{thebibliography}{A}

\bibitem[ChC1]{ChC1}
J. Cheeger and T.H. Colding,
Lower bounds on Ricci curvature and the
almost rigidity of warped products.  Ann. of Math. (2)  144  (1996),
no. 1, 189--237.

\bibitem[ChC2]{ChC2}
J. Cheeger and T.H. Colding,
On the structure of spaces with Ricci curvature
bounded below. I.  J. Differential Geom.  46  (1997),  no. 3, 406--480.

\bibitem[C]{C}
T.H. Colding,
Ricci curvature and volume convergence.
Ann. of Math. (2)  145  (1997),  no. 3, 477--501.

\bibitem[CN1]{CN1}
T.H. Colding and A. Naber, 
Sharp H\"older continuity of tangent cones for spaces with a lower Ricci curvature bound and applications, 
preprint, http://arxiv.org/abs/1102.5003.

\bibitem[CN2]{CN2}
T.H. Colding and A. Naber, 
Lower Ricci Curvature, Branching, and Bi-Lipschitz Structure of Uniform Reifenberg Spaces, preprint.

\bibitem[G]{G}
M. Gromov,
Metric structures for Riemannian and non-Riemannian spaces.
With appendices by M. Katz, P. Pansu and S. Semmes.  Birkh\"auser Boston, Inc., Boston, MA, 2007.

\bibitem[GLP]{GLP}
M. Gromov, J. Lafontaine, and P. Pansu,
Structures metriques pour les varieties riemanniennces.
Paris: Cedid/Fernand Nathan, 1981.

\bibitem[LV]{LV}
 J. Lott and C. Villani,
 Ricci curvature for metric-measure spaces via optimal transport,
 Ann. of Math. (2) 169 (2009) No. 3, 903--991.

\bibitem[M1]{M1}
X. Menguy,
Noncollapsing examples with positive Ricci curvature and infinite
topological type.
GAFA 10 (2000), no. 3, 600--627.

\bibitem[M2]{M2}
X. Menguy,
Examples of strictly weakly regular points.
GAFA
11 (2001), no. 1, 124--131.

\bibitem[M3]{M3}
X. Menguy,
Examples of manifolds and spaces with positive Ricci curvature,
Ph.D. thesis, Courant Institute, New York University 2000.

\bibitem[M4]{M4}
X. Menguy,
Examples of nonpolar limit spaces,
Amer. J. Math. 122 (2000), no. 5, 927--937.

\bibitem[Mo]{Mo}
J. Moser,
On the volume elements on a manifold, 
Trans. Amer. Math. Soc. 120 (1965), 286-294.

\bibitem[P1]{P1}
G. Perelman,
Construction of manifolds of positive Ricci curvature with big volume and
large Betti numbers. (English summary) Comparison geometry (Berkeley, CA,
1993--94), 157--163,
Math. Sci. Res. Inst. Publ., 30, Cambridge Univ. Press, Cambridge, 1997.

\bibitem[P2]{P2}
G. Perelman,
A complete Riemannian manifold of positive
Ricci curvature with Euclidean volume growth and nonunique asymptotic
cone.  Comparison geometry (Berkeley, CA, 1993--94),  165--166, Math.
Sci. Res. Inst. Publ., 30, Cambridge Univ. Press, Cambridge, 1997.

\bibitem[P3]{P3} 
G. Perelman, 
Alexandrov spaces with curvatures bounded from below II. 
preprint, 1991.

\bibitem[S]{S}
K.T. Sturm,
On the geometry of metric measure spaces; I,  
Acta Math. 196, 1 (2006), 65--131.

\end{thebibliography}
\end{document}